\documentclass[11pt,reqno]{amsart} 
\usepackage{setspace}
\usepackage{amsmath}
\usepackage{amscd,amsthm,amsfonts,amsopn,amssymb,verbatim}
\usepackage{mathrsfs}
\numberwithin{equation}{section}

\newtheorem{lemma}{Lemma}

\theoremstyle{remark}

\def\be{\begin{equation}}
\def\ee{\end{equation}}

\allowdisplaybreaks
\begin{document}

\title[Prescribing the binary digits of the primes, II]
{Prescribing the binary digits of the primes, II}
\date{\today}
\author{Jean Bourgain}
\address{Institute for Advanced Study, Princeton, NJ 08540}
\email{bourgain@math.ias.edu}
\thanks{This work was partially supported by NSF grants DMS-1301619 and DMS-0835373}

\begin{abstract}
We obtain the expected asymptotic formula for the number of primes
$p<N=2^n$ with $r$ prescribed (arbitrary placed) binary digits,
provided $r<{cn}$ for a suitable constant $c>0$.
This result improves on our earlier result where $r$ was assumed to
satisfy $r<c\big(\frac n{\log n}\big)^{4/7}$.
\end{abstract}

\maketitle

\section
{\bf Summary}

This paper is a follow up on \cite{B1}.
We establish the following stronger statement.

\medskip

\noindent{\bf Theorem.}
{\sl Let $N=2^n$, $n$ large enough, and $A\subset\{1, \ldots, n-1\}$ such that
\be\label{0.1}
r=|A|<cn
\ee
($c$ a suitable constant).
Then, considering binary expansions \hfill\break
$x=\sum_{j<n} x_j 2^j$ ($x_0=1$ and $x_j =0, 1$ for $1\leq j<n$) and assignments
$\alpha_j$ for $j\in A$, we have
\be\label{0.2}
|\{ p< N; \text { for } \ j\in A, \text { the $j$-digit of $p$ equals $\alpha_j$}\}|\sim 2^{-r} \frac N{\log N}.
\ee}

In \cite {B1}, the corresponding result was proven under the more restrictive condition
\be\label{0.3}
r< c\Big(\frac n{\log n}\Big)^{4/7}.
\ee
We also refer the reader to \cite{B1}, \cite{B2} for some background and motivation.
In particular the paper of Harman and Katai \cite{H-K} and complexity issues for the Moebius function and the primes,
raised by G.~Kalai, were seminal to those investigations.

The strategy followed here is roughly similar to the one in \cite{B1}, except for the fact that the additive Fourier
spectrum (together with Vinogradov's estimate) is only used to bound the contribution of the minor arcs.
In the treatment of the major arcs, we switch immediately to multiplicative characters (see \eqref{2.8} below)
and are led to study correlations of both the von Mangoldt function $\Lambda$ and the given function 
$f= 1_{[x<N; x_j=\alpha_j\text { for } j\in A]}$ with multiplicative characters.
This issue for $\Lambda$ is classical and depends on Dirichlet $L$-function theory.
We rely here on the same basic facts that were used in \cite{B1}.
As in \cite {B1}, we subdivide primitive characters $\mathcal X$ into two classes $\mathcal G$ and $\mathcal B$
(`good' and `bad') depending on the zero-free region of $L(s, \mathcal X)$.
It turns out that non-trivial bounds on the multiplicative spectrum of $f$ are only required if $\mathcal X\in\mathcal B$
(which is a small set of characters).
Here, we are again invoking the Gallagher-Iwaniec estimate on the improved zero-free region of $L(s, \mathcal X)$ for
$\mathcal X(\text{mod\,} q)$ with $q$ a power of 2, though the precise quantitative form of $[I]$ (which was responsible for the condition
\eqref{0.3}) is no longer relevant here.
Basically any statement that for $q$ as above, $L(s, \mathcal X)$ has a zero-free region $1-\sigma < c\frac {\log \log qT}
{\log q T}$, $|\gamma|<T$, where $\rho =\sigma+i\gamma$, would suffice for our purpose.
This fact ensures then that no character $\mathcal X\in\mathcal B$ has conductor which is a power of 2, which is essential to 
our analysis.
Note that possible Siegel zero's in any case forces us to introduce the class $\mathcal B$, even if $r$ were further reduced.
As in \cite {B1}, one needs to evaluate sums of the form
\be\label{0.4}
\sum_{x\in I, q_0|x} f(x)
\ee
\and
\be\label{0.5}
\sum_{x\in I, q_0|x} \mathcal X(x) f(x)
\ee
with $I\subset\{1, \ldots, N\}$ intervals of a certain size and $\mathcal X\in\mathcal B$, $\mathcal X(\text{mod\,} q)$, $(q, q_0)~=~1$.
The main technical innovation compared with \cite {B1} is a more efficient strategy to estimate these sums, leading to the
required information under less restrictive hypothesis on $r$.

The above theorem may be seen as a relative of Linnik's result on the least prime in an arithmetic progression.
One key difference is that possible Siegel zero's do not affect the final statement (though they play technically a
role in the argument).

Our presentation is completely self-contained, apart from basic number theoretic results and we will not refer to \cite {B1}.

\section
{\bf Preliminaries}

For $x\in \{1, 2, \ldots, 2^n-1\}$, write $x=\sum_{0\leq j<n} x_j 2^j$ with $x_j=0, 1$.

Let
$$
f(x) = 1_{[x<2^n; x_j=\alpha_j \text { for } j\in A]} \text { and } N=2^n
$$
where
$$
A= \{0= j_0< j_1< \ldots < j_r\}\subset \{0, 1, \ldots, n-1\}.
$$
We assume
\be\label {1.1}
r+1= |A|=\rho n
\ee
with $\rho>0$ bounded by a sufficiently small constant $c>0$.

For $\lambda \in \mathbb R$, denote
$$
\hat f(\lambda)= 2^{-n} \sum_{x=0}^{2^n-1} e^{2\pi i \lambda x} f(x) =2^{-|A|}\prod_{j\in A} 
e^{2\pi i\lambda \alpha_j 2^j}
\prod_{\substack{1\leq j<n\\ j\not\in A}} \frac {1+e^{i\pi \lambda 2^{j+1}}}2.
$$
Thus
\be\label {1.2}
|\hat f(\lambda)|= 2^{-|A|} \prod_{\substack{1\leq j<n\\j\not\in A}} |\cos \pi \lambda 2^j|.
\ee

\begin{lemma}\label{Lemma1}
\be\label{1.3}
2^{r+1} \sum_{k=0}^{2^n-1} \Big|\hat f\Big(\frac k{2^n}\Big)\Big| < 2^{C\rho(\log \frac 1\rho)n}
\ee
for some constant $C$.
\end{lemma}

\begin{proof}
By \eqref{1.2}, the left side of \eqref{1.3} equals
\be\label{1.4}
\sum^{2^n-1}_{k=0} \prod_{\substack{1\leq j<n\\ j\not\in A}} \Big|\cos\pi \frac k{2^{n-j}}\Big|.
\ee
Writing $k=k' +2^{n-j_1}\ell, 0\leq k'< 2^{n-j_1}, 0\leq \ell < 2^{j_1}$, we get
$$
\eqref{1.4} =\sum_{k'<2^{n-j_1}} \sum_{\ell< 2^{j_1}} \ \prod^{j_1-1}_{j=1} \Big|\cos \pi \Big(
\frac {k'}{2^{n-j}}+\frac \ell{2^{j_1-j}}\Big)\Big| \prod_{\substack{ j_1< j< n\\ j\not\in A}}
\Big| \cos\pi \frac {k'}{2^{n-j}}\Big|.
$$
Evaluate the inner sum with fixed $k'$ as  
$$
\begin{aligned}
&\sum_{\ell< 2^{j_1}} \prod_{j=1}^{j_1-1} \Big|\cos \pi \Big(\frac{k'}{2^{n-j}}+\frac \ell{2^{j_1-j}}\Big)\Big|\leq\\
&\max_\theta \sum_{\ell<2^{j_1}} \ \prod^{j_1-1}_{j=1} \Big|\cos \pi \frac {2^j(\ell+\theta)}{2^{j_1}}\Big|=\\
&\max_\theta\Big\{\sum_{\ell< 2^{j_1}} 2^{-j_1} \Big|\sum_{x=0}^{2^{j_1}-1} 
e^{2\pi i x2^{-j_1}(\ell+0)} \Big|\Big\}\leq\\
&\max_\theta \Big\{\sum_{\ell< 2^{j_1}} \ \frac 4{2^{j_1}\Vert\frac{\ell+\theta}{2^{j_1}} \Vert+1}\Big\} 
<Cj_1
\end{aligned}
$$
for some constant $C$.
Hence
$$
\eqref{1.4} < C(j_1-j_0) \sum_{k'<2^{n-j_1}} \ \prod_{\substack{j_1<j<n\\ j\not\in A}} \Big|\cos \pi \frac {k'}
{2^{n-j}}\Big|
$$
and we repeat the process with the $k'$-sum, replacing $n$ by $n-j_1$,
$j_1$ by $j_2-j_1$ etc.
It follows that
\be\label{1.5}
\eqref{1.4} < C^r \prod^r_{s=1} (j_{s+1} -j_s)
\ee
where we have set $j_{r+1}=n$.
Since $u\leq \frac 1\theta 2^{\theta u}$ for $u\geq 0, \theta\geq 0$,
$$
\eqref{1.5} <\Big(\frac C\theta\Big)^r 2^{\theta n} =\Big(\Big(\frac C\theta\Big)^\rho 2^\theta\Big)^n< 
2^{C\rho(\log \frac  1\rho)n}
$$
for an appropriate choice of $\theta$, proving \eqref{1.3}.
\end{proof} 
\medskip

\begin{lemma}\label{Lemma2}
\be\label{1.6}
2^r \int_0^1 |\hat f(\theta)|d\theta< 2^{C\rho(\log\frac 1\rho)n-n}.
\ee
\end{lemma}

\begin{proof} {}\hfill\break

Since $\hat f(\theta)$ is a trigonometric polynomial with spectrum in \hbox{$\{0, 1, \ldots, 2^n-1\}$,}
$$
\hat f(\theta) =2^{-n} \sum_{k=0}^{2^n-1} \hat f\Big(\frac k{2^n}\Big) D\Big(\theta -\frac k{2^n}\Big)
$$
with $D(\theta ) =\sum^{2^n-1}_{k=0} e^{2\pi i k\theta}$ the Dirichlet kernel.
It follows from Lemma 1 that
$$
2^r \int_0^1 |\hat f(\theta )|d\theta \leq \Vert D\Vert_1 \, 2^{C\rho(\log \frac 1\rho)n-n}\lesssim
n2^{C\rho(\log \frac 1\rho) n-n}
$$
proving \eqref{1.6}.
\end{proof}
\medskip

\begin{lemma}\label{Lemma3}
Let $Q< 2^{n/100}$. Then
\be\label {1.7}
2^r \sum_{\substack {q<Q, q\text{ odd}\\ 1\leq a<q, (a, q)=1}} |\hat f\Big(\frac aq\Big)\Big|< Q^{C\rho\log \frac 1\rho}.
\ee
\end{lemma}

\begin{proof}
Take $m$ such that
$$
2^{m-1} \leq Q^2<2^m.
$$
Clearly there is $0<j_*<n-2m$ so that the interval $I=\{j_*, \ldots, j_*+m-1\}$ satisfies
\be\label{1.8}
r'=|A'|=|A\cap I|< 2\rho m.
\ee
We note that
$$
2^r \Big|\hat f\Big(\frac aq\Big)\Big|\leq \prod_{j\in I\backslash A'} \Big|\cos \frac{\pi a2^j}q\Big|= 2^{r'} \Big|\hat g\Big(
\frac{a2^{j_*}}q\Big)\Big|
$$
where
$$
g= 1_{[x<2^m; x_j=\alpha_{j+j_*}\text { for } j\in A'-j_*]}.
$$
It follows then from \eqref{1.6},\eqref{1.8} that
\be\label{1.9}
2^{r'}\int_0^1 |\hat g(\theta)|d\theta< 2^{C\rho(\log \frac 1\rho)m-m}.
\ee
The set of points $\mathcal F=\Big\{\frac{a2^{j_*}} q (\text{mod\,} 1); q<Q, q \text { odd}, (a, q)=1\}$ are clearly pairwise $2^{-m}$-separated.
Write for $\xi\in \mathcal F$
$$
|\hat g(\xi)| \leq 2^m\int_{|\theta-\xi|< 2^{-m-1}} |\hat g(\theta)|d\theta+ 2^m \int_{|\theta-\xi|< 2^{-m-1}} |\hat g(\theta)
-\hat g(\xi)|d\theta
$$
and
$$
\begin{aligned}
\sum_{\xi\in\mathcal F} |\hat g(\xi)| &\leq 2^m \int_0^1 |\hat g|+\int_0^1 |(\hat g)'|\\
2^{r'} \sum_{\xi\in\mathcal F} |\hat g(\xi)| &\lesssim 2^{m+r'}\int_0^1 |\hat g|<2^{C\rho(\log \frac 1\rho)m}< Q^{2C\rho\log \frac 1\rho}
\end{aligned}
$$
where we used Bernstein's inequality and \eqref {1.9}.
This proves \eqref{1.7}
\end{proof}

For small $q$, there is the following individual bond.

\begin{lemma}\label{Lemma4}
Let $1<q<n^{\frac 1{10_\rho}}$ and odd, $(a, q)=1$.
Then
\be\label{1.10}
2^r \Big|\hat f\Big(\frac aq\Big)\Big| < 2^{-\sqrt n}.
\ee
\end{lemma}

\begin{proof}
Clearly
\be\label{1.11}
2^r \Big|\hat f\Big(\frac aq\Big)\Big| =\prod_{\substack{1\leq j<n\\ j\not\in A}} \Big|\cos \pi 2^j\frac aq\Big|
\leq \gamma^{\frac 12 \frac n\ell} \ \text { with } \ \ell =[^2 \log q]+1
\ee
where $\gamma$ is an upper bound on
\be\label{1.12}
\prod_{\substack{0\leq j<\ell\\ j\not\in E}} \Big|\cos \pi 2^j\frac{a'}q\Big|
\ee
where $(a', q)=1$ and $E\subset \{0, 1, \ldots, \ell-1\}$ satisfies $|E|< 2\rho\ell$.

Take $0\leq j_* \leq 2\rho\ell$ such that $j_*\not\in E$ and set $2^{-\ell'-1} \leq \Big\Vert 2^{j_*}\frac {a'}q\Big\Vert 
< 2^{-\ell'}, \hfill\break 0\leq \ell'\leq \ell$.

Then
$$
\Big\Vert 2^j\frac{a'}q\Big\Vert \sim 2^{j-j_*-\ell'} \text { for } 0\leq j-j_* < \ell'-1.
$$
If $\ell' > 10+2\rho\ell$, we can find $j$ such that $\ell'-2\rho\ell \leq j-j_*< \ell'-1$ and $j\not\in E$.

Then $\big\Vert 2^j \frac{a'} q\big\Vert\gtrsim 2^{-2\rho\ell}$.
Hence, in either case we find some $0\leq j<\ell$, $j\not\in E$, such that $\Vert 2^j\frac{a'}q\Vert\gtrsim 2^{-2\rho\ell}\gtrsim
q^{-2\rho}$.
It follows that
$$
\eqref{1.12} \leq \Big|\cos \pi 2^j\frac {a'}q\Big| < 1-\frac 12\Big\Vert\frac {2^j a'}q\Big\Vert^2 < 1-cq^{-4\rho}.
$$
Substituting in \eqref{1.11} gives the bound $e^{-\frac{cn}{q^{4\rho} \log q}}$ and the Lemma follows.
\end{proof}

\section
{\bf Minor Arcs contribution}

Let $N=2^n$.
Write
\be\label{2.1}
\sum_{k\leq N} \Lambda(k) f(k)= \int^1_0 S(\alpha) \overline S_f (\alpha) d\alpha
\ee
denoting
\be\label{2.2}
S(\alpha)= \sum\Lambda(k) e(k\alpha) 
\ee
and
\be\label{2.3}
S_f(\alpha)= \sum f(k) e(k\alpha)= N\hat f(\alpha).
\ee
We assume $f(k)=0$ for $k$ even, since obviously $k\equiv 1 (\text{mod\,} 2)$ is a necessary condition for $f$ to capture primes.

\medskip

We fix a parameter $B=B(n)$ which will be specified later, $B$ at most a small power of $N$.

The major arcs are defined by
\be\label{2.4}
\mathcal M(q, a)= \Big[\Big|\alpha -\frac aq\Big|<\frac B{qN}\Big] \ \text { where } \  q< B.
\ee
Given $\alpha$, there is $q<\frac NB$ such that
$$
\Big|\alpha-\frac aq\Big| <\frac B{qN}<\frac 1{q^2}.
$$

From Vinogradov's estimate  (Theorem 13.6 in [I-K])
\begin{align}\label {2.5}
|S(\alpha)| &< \Big(q^{\frac 12} N^{\frac 12}+q^{-\frac 12}N+N^{\frac 45}\Big)(\log N)^3\nonumber\\
&\ll \Big(\frac N{\sqrt B}+\frac N{\sqrt q}+N^{4/5}\Big) (\log N)^3.
\end{align}
Hence if $q\geq B$,
\be\label{2.6}
|S(\alpha)|\ll \frac N{\sqrt B} (\log N)^3.
\ee
Thus the minor arcs contribution in \eqref{2.1} is at most
\be\label{2.7}
\ll\frac N{\sqrt B}(\log N)^3 \Vert S_f\Vert_1.
\ee
Since by (2.16)
\be\label{2.8}
\Vert S_f\Vert_1 < 2^{-r} N^{C\rho\log \frac 1\rho}
\ee
we take
\be\label {2.9}
\log B> 3C\rho \Big(\log \frac 1\rho\Big)n
\ee
which takes care of the minor arcs contribution.

\medskip

\section
{\bf Major arcs analysis}

Next, we analyze the major arcs contributions $(q< B)$
\be\label{3.1}
\sum_{(a, q)=1} \ \int\limits_{|\alpha-\frac aq|<\frac B{qN}} S(\alpha) \overline S_f(\alpha) d\alpha.
\ee
Write $\alpha =\frac aq +\beta$. Defining
$$
\tau(\chi) =\sum^q_{m=1}  \chi(m) e_q(m)
$$
we have (see [D], p. 147)
\be\label {3.2}
S(\alpha) =\frac 1{\phi(q)} \sum_{\chi} \tau (\overline{\chi}) \chi(a) \Big[\sum_{k\leq N} \chi(k) \Lambda (k)
e(k\beta)\Big]+O\big((\log N)^2\big)
\ee
Assume $\chi$ is inducted by $\chi_1$ which is primitive $(\text{mod\,} q_1),q_1|q$.
Then from [D], p. 67
\be\label {3.3}
\tau (\bar \chi) =\mu \Big(\frac q{q_1}\Big) \bar\chi_1\big(\frac q{q_1}\Big)\tau (\bar\chi_1)
\ee
which vanishes, unless $q_2 =\frac q{q_1}$ is square free with $(q_1, q_2)=1$.

The contribution of $\chi$ in \eqref{3.1} equals
\be\label{3.4}
\frac{\tau(\overline{\chi})}{\phi (q)} \int\limits_{|\beta|<\frac B{qN}}\Big[\sum_{k\leq N} \chi(k) \Lambda(k) e(k\beta)\Big]
\Big[\sum_{k<N} f(k) \Big(\sum^q_{a=1} \chi(a) e_q(-ak)\Big) e(-k\beta)\Big] d\beta.
\ee
We have
\begin{align}\label{3.5}
\sum^q_{a=1} e_q(ak) \chi(a) &=\sum _{(a, q)=1} e_q(ak) \chi_1(a) \nonumber\\
&=\Bigg[\sum_{(a_1, q_1)=1} e_{q_1} (a_1k) \chi_1 (a_1)\Bigg] \Bigg[\sum_{(a_2, q_2)=1} e_{q_2} (a_2k)\Bigg]\mathcal X_1(q_2)
\nonumber\\
&= \overline{\chi_1}(k) \tau(\chi_1) c_{q_2} (k) \mathcal X_1(q_2) 
\end{align}
where
\be\label{3.6}
c_{q_2}(k) =\frac {\mu \Big(\frac{q_2}{(q_2, k)}\Big) \phi(q_2)}{\phi\Big(\frac{q_2}{(q_2, k)}\Big)}.
\ee

From \eqref{3.3}, \eqref {3.5}, \eqref{3.6}
\begin{align}\label {3.7}
\frac {\tau(\overline{\chi})}{\phi(q)}\Big[ \sum^q_{a=1}  \chi(a) e_q(-ak)\Big] &=\frac {|\tau({\chi_1})|^2}{\phi(q_1)} \
\frac {1}{\phi\Big(\frac {q_2}{(q_2, k)}\Big)}\,  \mu\big(( q_2, k)\big) 
\overline{\chi_1}(k)\nonumber\\
&=\frac {q_1}{\phi(q_1)} \ \frac {1}{\phi(\frac{q_2}{(q_2, k)})} \, \mu\big((q_2, k)\big)  \overline{\chi_1}(k)
\end{align}

Returning to \eqref{3.2}, rather than integrating
in $\beta$ over the interval $|\beta|<\frac B{qN}$, we introduce a weight function
$$
w\Big(\frac{qN}{B} \beta\Big)
$$
where $0\leq w\leq 1$ is a smooth bumpfunction on $\mathbb R$ such that $w=1$ on $[-1, 1]$, supp $w\subset [-2, 2]$ and
$$
|\hat w(y)|< C e^{-|y|^{1/2}}.
$$
\big(Note that this operation creates in \eqref{3.1} an error term that is captured by the minor arcs contribution \eqref{2.7}\big).

Hence, substituting \eqref{3.7}, \eqref{3.4}  becomes
\be\label{3.8}
\frac {q_1}{\phi(q_1)}\frac B{qN}  \sum_{k_1, k_2<N} \hat w\Big(\frac B{qN} (k_1-k_2)\Big)\mathcal X(k_1)\Lambda (k_1) 
f(k_2)\frac {\mu
\big((q_2, k_2)\big)}{\phi\big(\frac {q_2}{(q_2, k_2)}\big)} \overline{\mathcal X_1}(k_2)
\ee
and we observe that by our assumption on $\hat w$ the $k_1, k_2$ summation in \eqref{3.8} is restricted to $|k_1-k_2|< \frac {qN}B n^3$, 
up to a negligible error.

We first examine the contribution of the principal characters.

For $\mathcal X=\mathcal X_0, q_1=1$ and \eqref{3.8} becomes
\be\label{3.9}
\frac B{qN} \sum_{k_1, k_2<N} \hat w \Big(\frac B{qN} (k_1-k_2)\Big) \Lambda(k_1) f(k_2) 
\frac {\mu\big(( q, k_2)\big)}{\phi\big(\frac q{(q, k_2)}\big)}.
\ee
Fixing $k_2$, perform the $k_1$-summation in \eqref{3.9}.
Writing
\be\label{3.10}
\psi(x)=x- \sum_{\substack {\zeta(\rho)=0\\ |\gamma|< B^2}} \ \frac{x^\rho}{\rho} +O \Big(\frac x{B^2}(\log x)^2\Big)
\ee
for $x>\frac NB$ and assuming also
\be\label{3.11}
\log B< \frac n{1000}
\ee
partial summation, together with the usual zero-density and zero-free region estimate, give
\be\label{3.12}
\eqref{3.9} =\sum_{k\leq N} f(k) \frac {\mu((q, k))}{\phi (\frac q{(q, k)})}
\ee
\be\label{3.13}
+O\Big\{\Big[ \sum_{k\leq N} \frac {f(k)}{\phi(\frac q{(q, k)})}\Big] \exp \big(-(\log N)^{\frac 12}\big) \Big\}.
\ee

Let $\kappa(q)$ be a function satisfying the following
\smallskip

\noindent
{\bf Assumption A.} {\sl Let $q_0<B$ be odd and square free. Then
\be\label{3.14}
\sum_{q_0|k, k<N} f(k) =\mathbb E[f] \frac N{q_0} +O\Big(\kappa(q_0)  N\mathbb E[f]\Big)
\ee
where $\mathbb E[f]$ denotes the normalized average.}

\medskip

Assuming $q$ square-free (sf)  and  odd, \eqref{3.12} equals
$$
\begin{aligned}
&\sum_{q'|q} \, \frac {\mu(q')}{\phi\big(\frac q{q'})} \Big[\sum_{(q, k)=q'} f(k)\Big]=\\[6pt]
&\sum_{q'|q} \, \frac{\phi(q')}{\phi(q)} \mu(q') \sum_{q''|\frac q{q'}} \mu(q'')\Big[\sum_{\substack {q'q''|k\\ k<N}} f(k)\Big]
\end{aligned}
$$
and substituting \eqref{3.14} we obtain
\be\label{3.15}
N \, \mathbb E [f] \, \sum_{q'|q} \, \frac{\mu(q')}{\phi(q/q')} \, \sum_{q''|\frac q{q'}} \, \frac {\mu(q'')}{q'q''}
\ee
\be\label{3.16}
+O\Big( N\, 2^{-r} \sum_{q'|q} \, \frac {\phi(q')}{\phi(q)} \, \sum_{q''|\frac q{q'}} \,  {\kappa (q'q'')}\Big).
\ee

Next
\begin{align}\label{3.17}
\eqref{3.15} &= N\, \mathbb E[f] \, \sum_{q'|q} \, \frac {\mu(q')}{\phi (q/q')} \, \frac 1{q'} 
\prod_{p|\frac q{q'}} \, \Big(1-\frac 1p\Big)\nonumber \\[6pt]
&=2N\mathbb E[f] \sum_{q'|q} \frac {\mu(q')}{q}\nonumber\\
&=\begin{cases} N\, \mathbb E[f] \text { if }  \ q=1\\ 0 \ \text { otherwise.}\end{cases}
\end{align}

Summing \eqref{3.16} over $q<B$ sf and odd, we have the estimate (setting $q_1=q'q''$)
\begin{align}\label{3.18}
&2^{-r} N \sum_{\substack{q_1< B\\ q_1 \, \text{sf}}} \kappa (q_1) \Big[\sum_{\substack{ q''|q_1, (q_1, q_2)=1\\ q_2< B \, \text{sf}}}
\, \frac 1{\phi(q'')\phi (q_2)}\Big]\nonumber\\
&<2^{-r} N(\log B)^2 \Big[\sum_{\substack{q<B\\ q\, \text{sf,  odd}}}\ \kappa(q)\Big].
\end{align}
For $q$ sf and even, set $q=2q_1$ and note that \eqref{3.12} equals
$$
\sum_{q'|q_1} \ \frac{\mu(q')}{\phi\big(q_1/q')} \ \Big[\sum_{(q_1, k)=q'} f(k)\Big]
$$
and we proceed similarly as above with the same conclusion and $q$ replaced by $q_1$. 

The first factor in \eqref{3.13} contributes for
\begin{align}\label{3.19}
&\sum_{q<B, q\,\text{sf}} \, \sum_{q'|q, q' \text { odd}} \, \frac 1{\phi(\frac q{q'})} \sum_{\substack{k\leq N\\ q'|k}} f(k)
\overset {\eqref{3.14}}\leq \nonumber\\
& N2^{-r} \sum_{\substack {q', q''<B\\ q', q'' \text{ sf}, \, q'\text { odd}}} \ \frac 1{\phi(q'')} \Big(\frac 1{q'}+\kappa (q')\Big)<\nonumber\\
& N2^{-r} \Big[ (\log B)^2 +(\log B) \Big(\sum_{\substack {q<B\\ q \text{\,sf,  odd}}} \kappa(q)\Big)\Big]
\end{align} 

Thus, from the preceding, the contribution of the principal characters equals 
\be\label{3.20}
2\mathbb E[f] N+CN2^{-r} n^2\Big\{\substack{\sum\limits_{\substack{q< B\\ q \, \text{sf,  odd}}}}{\kappa(q)} \Big\} + CN2^{-r} n^2 \exp(-n^{\frac 12}).
\ee
Next, consider non-principal characters, i.e. $q_1>1$.  

\stepcounter{equation}
\stepcounter{equation}

Estimate \eqref{3.8} by
\be\label{3.23}
\frac {q_1}{\phi(q_1)} \ \frac {B^3}N (4.21).(4.22) + O\Big({\frac N{q B^2}}\Big)
\ee
with

$$
(4.21) =\max_{|I|\sim\frac N{B^3}} \Big| \sum_{k\in I} \mathcal X(k)\Lambda (k)\Big|
$$
and
$$
(4.22) = \sum_I \Big|\sum_{k\in I} f(k) \ \frac {\mu((q_2, k))}{\phi(\frac {q_2}{(q_2, k)})} \bar{\mathcal X_1} (k)\Big|
$$
where $I$ runs over a partition in intervals of size $\sim\frac N{B^3}$.

Obviously
$$
(4.22) \leq \sum_{k\leq N} \ \frac {f(k)}{\phi(\frac {q_2}{(q_2, k)})}
$$
and summing over $q_2 <B, q_2$  sf, gives the estimate \eqref{3.19}.

The factor (4.21) is bounded by
\be\label{3.24}
\max_{\frac N{B^2} <x<N} |\psi (x+h, \mathcal X) -\psi (x, \mathcal X)| \text { with } h\sim \frac N{B^3}.
\ee
Choose a parameter $B<T<N^{\frac 1{100}}$ and denote by $N(\alpha, T; \mathcal X)$ the number of zeros of $L(s, \mathcal X)$ such that
$$
\alpha\leq \sigma\leq 1, |t|\leq T \quad (s=\sigma+it).
$$
Then (see \cite{Bom}, Theorem 14)
\be\label{3.25}
N(\alpha) =\sum_{q\leq Q} \quad \sideset{}{^*}\sum_{ \chi(\text{mod\,} q)} \ N(\alpha, T; \chi)\ll (TQ)^{8(1-\alpha)}
\ee
where $\sum^*$ refers to summation over primitive characters.

Let $\chi$ be a non-principal character.
From Proposition 5.25 in [I-K], for $T\leq x$
\be\label{3.26}
\psi (x,  \chi)= - \sum_{\substack {L(\rho,  \chi)=0\\ |\gamma|\leq T}} \ \frac {x^\rho -1}\rho + O\Big(\frac xT(\log x q)^2\Big)
\ee
where $\rho=\beta+ i\gamma$. We denote
$$
\eta =\eta(\chi) =\min (1-\beta)
$$
with min taken over all zero's $\rho$ of $L(s, \chi)$ with $|\gamma|\leq T$.

Taking $T=B^5$, we get from \eqref{3.26} that
\be\label{3.27}
\eqref{3.24} \leq h\sum_{\substack{L(\rho, \mathcal X_1)=0\\ |\gamma| < B^5}} \ \frac 1{x^{1-\beta}} + O\Big(\frac {Nn^2}{B^5}\Big).
\ee
At this point, we fix some $\eta^* =\eta^*(n)$ and subdivide the primitive characters $\mathcal X_1$ in classes $\mathcal G$ and $\mathcal B$
depending on whether $\eta\geq\eta^*$ or $\eta<\eta^*$.

Recall that $q\leq B$.

Summing \eqref{3.27} over $q_1$, $\mathcal X_1(\text{mod\,} q_1)$ primitive and $\mathcal X_1\in\mathcal G$, we obtain from the density 
bound \eqref{3.25}.
\begin{align}\label {3.28}
\sum_{\mathcal X_1\in\mathcal G} \frac 1{x^{1-\beta}} = -2\int_{\frac 12}^{1-\eta_*} \frac 1{x^{1-\alpha}} dN(\alpha) 
&\leq 2 x^{-\frac 12} N\Big(\frac 12\Big)+ 2\log x\int^{\frac 12}_{\eta_*} \Big(\frac {B^{48}}x\Big)^\tau d\tau\nonumber\\
&< \frac {B^{21}}{N^\frac 12} + N^{-\frac 12\eta_*} <O(N^{-\frac 12\eta_*}).
\end{align}
Hence, the contribution to \eqref{3.23} of the $\mathcal X_1\in \mathcal G$ may be estimated by
\begin{align}\label{3.29}
&(\log\log B)\frac {B^3}N \Big[\frac N{B^3} N^{-\frac 12\eta_*}+ \frac {Nn^2}{B^4}\Big]. \eqref{3.19}<\nonumber\\
& n^3(N^{-\frac 12\eta_*}+n^2 B^{-1})\Big(1+ \sum_{q<B} \kappa(q)\Big) 2^{-r}N.
\end{align} 
Next, we consider the contribution of the $\mathcal X_1\in\mathcal B$.
Again from \eqref{3.25}
\be\label{3.30}
|\mathcal B|\ll (TB)^{8\eta_*}\leq B^{48\eta_*}.
\ee
Use the trivial bound $\frac N{B^3}$ on (4.21) for $\mathcal X_1\in\mathcal B$.
We get the following estimate for the $\mathcal B$-contribution to the first term of \eqref{3.23}
\be\label{3.31}
\sum_{\mathcal X_1 \in \mathcal B} \ \sum_{\substack{q_2<B\\ q_2 \text{\,sf}}} (4.22).
\ee
Introduce another parameter $\alpha (q_1, q_0)$ satisfying the condition.

\smallskip
\noindent
{\bf Assumption B.}

{\sl Given $q_0, q_1<B, (q_0, q_1)= 1$ with $q_0$ sf and odd, $\chi_1(\text{mod\,} q_1)$ primitive, $\chi_1\in\mathcal B$
\be\label {3.32}
\Big|\sum_{k\in I, q_0|k} f(k)\chi_1 (k)\Big|< \alpha(q_1, q_0) \Big[\sum_{k\in I} f(k) +|I| 2^{-r}\Big]
\ee
holds, whenever $I\subset [1, N]$ is an interval of size $\sim \frac N{B^3}$.}

\medskip

Hence
$$
\begin{aligned}
(4.22) &= \sum_I \sum_{\substack {q_2'|q_2\\ q_2' \text { odd}}} \ \frac 1{\phi(\frac {q_2}{q_2'})} \Big|\sum_{k\in I, (k, q_2) =q_2'}
f(k) \chi_1 (k)\Big|\\
&\leq N2^{-r} \sum_{\substack{q_2'|q_2\\ q_2' \text { odd}}} \ \frac 1{\phi(\frac {q_2}{q_2'})} \ \sum_{q_2''|\frac {q_2}{q_2'}, q_2'' \text { odd}}
\ \alpha (q_1, q_2' q_2'')
\end{aligned}
$$
and summation over sf $q_2< B$ gives
\begin{align}\label{3.33}
& N2^{-r}\sum_{\substack {q_3<B\\ q_3 \text{\,sf}, \text { odd}}} \alpha (q_1, q_3) 
\sum_{\substack{q_2''|q_3, q_2'''<B\\ q_2''' \text{\, sf}, (q_2''', q_3)=1}} \
\frac 1{\phi (q_2'') \phi (q_2''')}\nonumber\\
&\lesssim n^2 N2^{-r}\Big[\sum_{\substack {q_3<B\\ q_3 \, \text{sf,  odd}}} \ \alpha (q_1, q_3)\Big].
\end{align}

By \eqref{3.30}, this gives the following bound on \eqref{3.31}
\be\label{3.34}
n^2 B^{48 \eta_*} N2^{-r} \max_{\mathcal X\in \mathcal B}\Big[\sum_{q_0<B, q_0 \text{\,sf,  odd}} \ \alpha (q_1, q_0)\Big].
\ee
\medskip

In the next section, we will establish bounds on
\be\label{3.35}
\sum_{\substack {q<B\\ q \, \text{sf,  odd}}} \ \kappa (q)
\ee
and
\be\label{3.36}
\sum_{\substack{ q_0<B\\ q_0 \, \text{sf,  odd}}} \alpha (q_1, q_0).
\ee
In particular, \eqref{3.35} $<O(1)$ so that a choice
\be\label{3.37}
\eta_*=O\Big(\frac {\log n}n\Big)
\ee
suffices for \eqref {3.29} to be conclusive.

It is important to note that for this choice of $\eta_*$, no $\mathcal X_1\in\mathcal B$ has a conductor $q_1$ which is a power of 2.
Indeed, recalling the Gallagher-Iwaniec result (see \cite {H-K}, Lemma 5), if $\mathcal X_1$ is primitive $(\text{mod\,} 2^m)$, we obtain 
the following improved zero-free region
\be\label{3.38}
\eta (\mathcal X_1)\gtrsim [(\log 2^m T)(\log\log 2^m T)]^{-\frac 34} \sim (\log B\log\log B)^{-\frac 34} >\eta_*
\ee
(recall also that Siegel zeros are not a concern).

\section
{\bf Further estimates}

It remains to obtain suitable bounds on $\kappa(q)$ and $\alpha(q_1, q_0)$ introduced in Assumptions A and B.

Write for $q<B$, $q$ sf and odd
$$
\sum_{q|k} f(k) =\frac Nq \mathbb E[f] +\frac 1q \sum_{a=1}^{q-1} \sum_k f(k) e\Big( \frac{ak} q\Big)
$$
and hence
\be\label{4.1}
\kappa (q) \leq \frac{2^r}q \sum_{a=1}^{q-1} \Big|\hat f\Big(\frac aq\Big)\Big|.
\ee
It follows that \eqref{3.35} may be bounded by
\be\label{4.2}
2^r \sum_{\substack{ 1<q<B\\ q \, \text{sf, odd}}} \frac 1q \sum^{q-1}_{a=1} \Big|\hat f\Big(\frac aq\Big)\Big| \leq
\log B\sum_{\substack{ 1<q<B\\ q\text { sf, odd}\\ (a, q)=1}} \frac{2^r}q \Big|\hat f\Big(\frac aq\Big)\Big|.
\ee
Consider dyadic ranges $q\sim Q<B$.
Lemma \ref{Lemma3} provides an estimate $Q^{C\rho(\log \frac 1\rho)-1}$ $< Q^{-\frac 12}$ (for $\rho$ small enough) 
for the corresponding
contribution to the sum in the r.h.s. of \eqref{4.2}, while, for $Q<n^{\frac 1{10\rho}}$, Lemma \ref{Lemma4} gives an estimate
$Q2^{-\sqrt n}$.
It follows that
\be\label{4.3}
\eqref{4.2} < n.n^{-\frac 1{20\rho}}+n^{\frac 1{10\rho}+1} e^{-\sqrt n}< 2. n^{-\frac 1{20\rho}+1}.
\ee
Next, consider Assumption B.
Observe first that we can assume (by subdivision) $I$ to be of the form $[0, 2^m-1]+u2^m$ with $m=[\frac n2]$ say.

Fix $u\in\{0, 1, \ldots, 2^{n-m}-1\}$ such that $u_j=\alpha_{j+m}$ for $j+m\in A$ and define
\be\label{4.4}
f_1(x)=f(x+u2^m) \text { for } x\in \{0, \ldots, 2^m-1\}.
\ee
Thus
$$
f_1 =1_{\big[x< 2^m; x_j=\alpha_j \text { for } j\in A\cap [1, m-1]\big]}.
$$
It clearly suffices to establish inequalities \eqref{3.32} with $f|_I$ replaced by $f_1$, provided $\mathcal X_1(k)$ is replaced by
$\mathcal X_1(k+u2^m)$.
This basically leads to evaluate
\be\label{4.5}
\sum_{k<N, k+b\equiv 0 (\text{mod\,} q_0)} f(k) \mathcal X_1 (k+b)
\ee
without taking the restriction $k\in I$ into consideration.

Let us first assume $q_1>1$ is odd.
Write \eqref{4.5} as
\be\label{4.6}
\frac 1{q_0} \sum_{a_0=0}^{q_0-1} \sum_k f(k) e\Big(\frac{a_0}{q_0}(k+b)\Big) \mathcal X_1 (k+b)=\frac N{q_0} \sum_{a_0=0}^{q_0-1}
\sum_{(a_1, q_1)=1} \hat{\mathcal X_1} (a_1) e\Big(b\Big(\frac{a_0}{q_0}+\frac {a_1}{q_1}\Big)\Big) \hat f\Big(\frac {a_0}{q_0}+
\frac{a_1}{q_1}\Big)
\ee
with
$$
\hat{\mathcal X}_1 (a_1) =\frac 1{q_1} \sum_{x=0}^{q_1-1} \mathcal X_1 (x) e\Big(-\frac {xa_1}q\Big).
$$
Hence
$$
|\eqref{4.6}|\leq \frac N{q_0\sqrt{q_1}} \ \sum_{a_0=0}^{q_0-1} \ \sum_{(a_1, q_1)= 1} |\hat f\Big(\frac{a_0}{q_0}+\frac{a_1}{q_1}\Big)\Big|
$$
and summing over $1\leq q_0 <B, q_0$ sf, odd, $(q_0, q_1)=1$, we obtain a bound
$$
\frac N{\sqrt{q_1}} \log B\sum_{\substack{ 1\leq q_0<B, q_0 \text { sf, odd}, (q_0, q_1)=1\\
(a_0, q_0)=1, (a_1, q_1)=1}} \ \frac 1{q_0} \Big|\hat f\Big(\frac {a_0}{q_0}+\frac {a_1}{q_1}\Big)\Big|
$$
\be\label{4.7}
\leq N. n\ \sum_{\substack {1<q<B^2, q \text { sf, odd}\\ (a, q)=1}} \ \frac 1{\sqrt q} \Big|\hat f\Big(\frac aq\Big)\Big|.
\ee
By a similar estimate as used for \eqref{4.2}, we get for $\rho$ small enough
$$
\eqref{4.3} < n^{-\frac 1{30\rho}+1} N.2^{-r}
$$
and a bound
\be\label{4.8}
\eqref{3.36}< n^{-\frac 1{30\rho}+1}.
\ee 
If $q_1$ is even, write $q_1 =2^\nu q_1'$, $(q_1', 2)=1$ and $q_1' >1$ since $q_1$ is not a power of 2.
Let $\mathcal X_1=\mathcal X_0\mathcal X_1'$ with $\mathcal X_0(\text{mod\,} 2^\nu)$ and $\mathcal X_1'$ primitive $(\text{mod\,} q_1')$.
Write $k=z+2^\nu x$ with $z\in \{0, 1, \ldots, 2^\nu -1\}$, $x< 2^{n-\nu}$ and
\be\label{4.9}
\eqref{4.5} 
=\sum_{z=0}^{2^\nu -1} \mathcal X_0 (b+z) \sum_{\substack{x<2^{n-\nu}\\ b+z+2^\nu x\equiv 0(\text{mod\,} q_0)}}
\mathcal X_1' (b+z+2^\nu x)f_z(x)
\ee
denoting
\be\label{4.10}
f_z(x)= f(z+2^\nu x).
\ee
Thus
\be\label{4.11}
|\eqref{4.9}|\leq \sum_{z=0}^{2^\nu -1} \max_{b'} \Big|\sum_{\substack{x< 2^{n-\nu}\\ x+b' \equiv 0(\text{mod\,} q_0)}}
\mathcal X_1 (b'+x) f_z(x)\Big|.
\ee
Estimate the inner sum in \eqref{4.11} similarly to \eqref{4.5}, with $f$ replaced by $f_z, q_1$ by $q_1', 
N$ by $2^{-\nu} N$.
This gives a bound
\be\label{4.12}
2^{n-\nu}\mathbb E[f_z]. (n-\nu)^{-\frac 1{30\rho}+1}.
\ee
Summation of \eqref{4.12} over $z< 2^\nu$ implies that
$$
\eqref{4.11} \lesssim n^{-\frac 1{30\rho}+1} N2^{-r}
$$
so that \eqref{4.8} holds in general.

Summarizing, we proved that
\be\label{4.13}
\sum_{\substack{q<B\\ q \text { sf, odd}}} \kappa(q) \lesssim n^{-\frac 1{20\rho}+1}
\ee
and also
\be\label{4.14}
\sum_{\substack{q_0<B, q_0 \text { sf, odd}\\ (q_0, q_1)=1}} \alpha(q_1, q_0)\lesssim n^{-\frac 1{30\rho}+1}.
\ee
\bigskip

\section
{Conclusion}

Recalling \eqref{3.20}, \eqref{3.23}, \eqref{3.29}, \eqref{3.34} and inserting the estimates \eqref{4.13}, \eqref{4.14}, we
find that
\begin{align}
\label{6.1}
&\sum _{x<N} \Lambda(x) f(x)= 2\mathbb E[f]N+\nonumber\\
&N\mathbb E[f] O(n^{-\frac 1{20\rho}+3} +n^2 e^{-\sqrt n}+ 2^r B^{-1} +n^3 N^{-\frac 12\eta_*}+n^5B^{-1} +n^{-\frac 1{30\rho}+3}
B^{48\eta_*})
\end{align}

Recall also conditions \eqref{2.9}, \eqref{3.11} on $B$.

It remains to choose $\eta_*\sim \frac{\log n}n$ appropriately and let $\rho$ be small enough to conclude the Theorem.

\end{document}